\documentclass{article}
\usepackage{amsmath,amsthm,amsfonts}
\usepackage{soul}
\usepackage[colorlinks,linkcolor=blue,anchorcolor=blue,citecolor=blue]{hyperref}
\usepackage[numbers,sort&compress]{natbib}
\numberwithin{equation}{section}

\newtheorem{definition}{Definition}[section]
\newtheorem{theorem}[definition]{Theorem}
\newtheorem{proposition}[definition]{Proposition}
\newtheorem{remark}[definition]{Remark}
\newtheorem{lemma}[definition]{Lemma}
\newtheorem{corollary}[definition]{Corollary}

\title{A generalization of the Schwarz lemma for transversally harmonic maps\footnote{Supported by NSFC Grants No.11771087 and No.12171091.}}
\author{Xin Huang, Weike Yu}
\date{}

\begin{document}
\maketitle
	
\begin{abstract}
In this paper, we consider transversally harmonic maps between Riemannian manifolds with Riemannian foliations. In terms of the Bochner techniques and sub-Laplacian comparison theorem, we are able to establish a generalization of the Schwarz lemma for transversally harmonic maps  of bounded generalized transversal dilatation. In addition, we also obtain a Schwarz type lemma for transversally holomorphic maps between K\"ahler foliations.
\end{abstract}

\section{Introduction}
	

As we all know, the Schwarz lemma plays an important role in differential geometry and complex analysis. The classical Schwarz-Pick lemma (cf.\cite{ref20}) states that any holomorphic map between two unit discs $D$ in $\mathbb{C}$ is distance-decreasing with respect to the Poincar\'e metric. In \cite{[Ah]}, Ahlfors extended this result to holomorphic maps from the unit disc $D$ into a 1-dimensional K\"ahler manifold with curvature bounded above by a negative constant. Later, this lemma has been generalized to the holomorphic maps between higher dimension (almost) complex manifolds, see \cite{ref21, ref22, ref23, ref14, [To], [Ni], [CN], [Yu]} and an excellent book \cite{[KL-2]}, etc. Among them, a major advance is Yau's generalization, which states that any holomorphic map from a complete K\"ahler manifold with Ricci curvature bounded below by a non-positive constant $-k_1$ to a Hermitian manifold with holomorphic bisectional curvature bounded above by a negative constant $-k_2$ is distance-decreasing up to the constant $k_1/k_2$. Besides, it was also extended to generalized holomorphic maps between pseudo-Hermitian manifolds (cf. \cite{[DRY]}), and generalized holomorphic maps between pseudo-Hermitian manifolds and Hermitian manifolds (cf. \cite{[CDRY]}, \cite{[RT]}).
	
On the other hand, in \cite{ref16}, Goldberg and Har'El established a Schwarz type lemma for harmonic maps with bounded dilatation between Riemannian manifolds. Later, Shen \cite{[She]} introduced the concept of the harmonic maps with generalized dilatation by a constant, and proved a related Schwarz type lemma, which improves the Goldberg and Har'El's result. More specifically, he proved that every harmonic map with generalized dilatation by a constant $\beta>0$ from a complete Riemannian manifold with Ricci curvature bounded below by a non-positive constant $-k_1$ to a Riemannian manifold with sectional curvature bounded above by a negative constant $-k_2$ is distance-decreasing up to the constant $\beta^2 k_1/k_2$. Recently, some Schwarz type lemmas have also established for generalized harmonic maps between Riemannian manifolds, see \cite{[CZ17]}, \cite{[CLQ]}, etc.

The main aim of this paper is to establish Schwarz type lemmas for transversally harmonic maps between Riemannian manifolds with Riemannian foliations (see Definition \ref{def2.8}) and transversely holomorphic maps between Riemannian manifolds with K\"ahler foliations (see Definition \ref{def2.11}), which generalize Shen and Yau's Schwarz lemmas, respectively. To state our main results, let us introduce some notations. Given a Riemannian manifold $(M^m, g)$ with a Riemannian foliation $\mathcal{F}$,  let $\mathcal{V}=T\mathcal{F}$ denote the tangent bundle of $\mathcal{F}$, and let $\mathcal{H}$ be the orthogonal complement of $\mathcal{V}$ with respect to $g$. Using the Riemannian connection $\nabla^M$ on $(M^m,g)$, one can define two tensors $A$ and $h$ of type $(1,2)$ by
\begin{equation}
\begin{aligned}
A(X,Y) & =-\pi_{\mathcal{V}}(\nabla^M_{\pi_{\mathcal{H}}(Y)} \pi_{\mathcal{H}}(X)), \\  
h(X,Y) & =\pi_{\mathcal{H}}(\nabla^M_{\pi_{\mathcal{V}}(Y)}\pi_{\mathcal{V}}(X)),
\end{aligned}
\end{equation}
where $\pi_{\mathcal{H}}:TM\rightarrow\mathcal{H}$ and $\pi_{\mathcal{V}}:TM\rightarrow\mathcal{V}$ are the natural projections. Note that $A\equiv0$ if and only if the horizontal distribution $\mathcal{H}$ is integral, and the foliation $\mathcal{F}$ is said to be totally geodesic if $h\equiv0$. Moreover, we give the definitions of $C^1$-norms of $A$ and $h$ as follows:
\begin{equation}
\begin{aligned}
&|A|_{C_1}=\sup_{y \in M}\{|A|(y),|\nabla A|(y)\}, \\
&|h|_{C_1}=\sup_{y \in M}\{|h|(y),|\nabla h|(y)\}.
\end{aligned}
\end{equation}
For our purpose, we need to derive the Bochner formulas for these maps, and establish a sub-Laplacian comparison theorem on Riemannian manifolds with Riemannian foliations (see Sec.\ \ref{sec3}). In terms of the Bochner formulas and comparison theorem, we prove the following 
\begin{theorem}\label{thm1.1}
Let $(M^m,\mathcal{F},g)$ be a complete Riemannian manifold with a Riemannian foliation $\mathcal{F}$ and its transverse Ricci curvature bounded from below by $-K_1\le 0$, and $|A|_{C^1}, |h|_{C^1}$ bounded from above. Let $(N^n,\tilde{\mathcal{F}},\tilde{g})$ be another Riemannian manifold with a Riemannian foliation $\tilde{\mathcal{F}}$ and its transverse sectional curvature bounded from above by $-K_2<0$. Then for any transversally harmonic map $f:M \to N$ with bounded generalized transversal dilatation of order $\beta$ (cf. Definition \ref{def4.1}), we have 
\begin{equation}
f^*\tilde{g}_{\tilde{\mathcal{H}}}\le \beta^2 \frac{K_1}{K_2}  g_{\mathcal{H}},
\end{equation} 
 where $\mathcal{H}$ (resp. $\tilde{\mathcal{H}}$) is the orthogonal complement of the tangent bundle of $\mathcal{F}$ (resp. $\tilde{\mathcal{F}}$) with respect to $g$ (resp. $\tilde{g}$), and $g_{\mathcal{H}}=g|_{\mathcal{H}}$, $\tilde{g}_{\tilde{\mathcal{H}}}=\tilde{g}|_{\tilde{\mathcal{H}}}$. In particular, if $K_1=0$, $f$ is horizontally constant, i.e., $df_{{\mathcal{H}},\tilde{\mathcal{H}}}= \pi_{\tilde{\mathcal{H}}} \circ df \circ i_{\mathcal{H}}\equiv0$, where $i_{\mathcal{H}} : \mathcal{H} \to TM$ is the inclusion morphism.
\end{theorem}

\begin{theorem} \label{thm1.2}	
Let $(M^m,\mathcal{F},g,J)$ be a complete Riemannian manifold with a K{\"a}hler foliation $\mathcal{F}$ and its transverse Ricci curvature bounded from below by $-K_1 (K_1 \ge 0)$, $|A|_{C_1}$ and $|h|_{C_1}$ bounded from above. Let $(N^n,\tilde{\mathcal{F}},\tilde{g},\tilde{J})$ be another Riemannian manifold with a K{\"a}hler foliation $\tilde{\mathcal{F}}$ and its transverse bisectional curvature bounded from above by $-K_2 < 0$. Then any transversally holomorphic map $f: M \to N$ satisﬁes
\begin{equation}
f^*\tilde{g}_{\tilde{\mathcal{H}}}\le \frac{K_1}{K_2}  g_{\mathcal{H}}.
\end{equation} 
In particular, if $K_1=0$, $f$ is horizontally constant.
\end{theorem}

	\section{Transversally harmonic maps} \label{sec2}
	
	\subsection{Riemannian and K{\"a}hler foliations} \label{sec2.1}
	
	Let $(M^m,g)$ be a smooth Riemannian manifold and $\mathcal{V}\subset TM$ be an integrable distribution on $M$. The collection of integral manifolds of $\mathcal{V}$ is referred to as a foliation denoted by $\mathcal{F}$, where $\mathcal{V}$ is called the tangent bundle of $\mathcal{F}$, and the maximal connected integral submanifold of $\mathcal{V}$ is called a leaf of $\mathcal{F}$ (cf. \cite{ref1}, \cite{ref5}). Let $\mathcal{H}=\mathcal{V}^{\bot}$ be the orthogonal complement of $\mathcal{V}$ with respect to $g$, then we have the following orthogonal decomposition: 
\begin{equation}
	TM = \mathcal{V} \oplus \mathcal{H}.\label{2.1}
\end{equation} 
Note that $\mathcal{H}$ (resp. $\mathcal{V}$) is also referred to as the horizontal distribution (resp. vertical distribution) on $(M^m,g,\mathcal{F})$. A foliation is said to be Riemannian if $\mathcal{L}_{U}g_\mathcal{H}$ vanishes horizontally for all $U\in \mathcal{V}$, where $\mathcal{L}$ denotes the Lie derivative, $g_{\mathcal{H}}(\cdot, \cdot)=g(\pi_{\mathcal{H}}(\cdot), \pi_{\mathcal{H}}(\cdot))$, and $\pi_{\mathcal{H}}:TM\rightarrow\mathcal{H}$ is the natural projection. Let $\mathcal{Q}= TM/ \mathcal{V}$ be the normal bundle of $\mathcal{F}$, then there is a bundle map $\sigma:\mathcal{Q} \to \mathcal{H}$ satisfying $\pi_Q \circ \sigma=id$, where $\pi_Q:TM \to \mathcal{Q}$ is a projection. Let 
\begin{equation}
	g_\mathcal{Q} (s,t) = g(\sigma(s),\sigma(t)), \quad \forall s,t\in \Gamma \mathcal{Q},
\end{equation} 
which is a Riemannian metric on $Q$, then $\sigma:(\mathcal{Q},g_\mathcal{Q})\to (\mathcal{H},g_\mathcal{H})$ is an isomorphism. Without confusion, we always identify $\mathcal{H}$ and $\mathcal{Q}$ in the sense of above isomorphism.

On a Riemannian manifold $(M, g)$ with a Riemannian foliation $\mathcal{F}$, we introduce the generalized Bott connection ${\nabla}^{\mathfrak{B}}$ on $M$ (cf. \cite{ref27}), defined by 
\begin{equation} \label{bott}
	{\nabla}^{\mathfrak{B}}_XY =
	\begin{cases}
	\pi_{\mathcal{H}}[X,Y] & X \in \Gamma(\mathcal{V}),Y \in \Gamma(\mathcal{H})  \\ 
	\pi_{\mathcal{H}} \nabla^M_X Y & X \in \Gamma (\mathcal{H}),Y \in \Gamma(\mathcal{H})\\
	 \pi_{\mathcal{V}}[X,Y] & X \in \Gamma(\mathcal{H}),Y \in \Gamma(\mathcal{V})
	\\ \pi_{\mathcal{V}} \nabla^M_X Y & X \in \Gamma (\mathcal{V}),Y \in \Gamma (\mathcal{V})
	\end{cases},
\end{equation}
where $\nabla^M$ is the Riemannian connection on $M$, $\pi_{\mathcal{H}}:TM\rightarrow\mathcal{H}$ and $\pi_{\mathcal{V}}:TM\rightarrow\mathcal{V}$ are natural projections. Clearly,  ${\nabla}^{\mathfrak{B}}$ preserves the decomposition \eqref{2.1}, and satisfies
\begin{align}
\nabla^{\mathfrak{B}}_Xg_{\mathcal{H}}=0,\ \nabla^{\mathfrak{B}}_Yg_{\mathcal{V}}=0,
\end{align}
for any $X\in TM,\ Y\in \mathcal{V}$, where $g_{\mathcal{H}}=g(\pi_{\mathcal{H}}(\cdot), \pi_{\mathcal{H}}(\cdot))$ and $g_{\mathcal{V}}=g(\pi_{\mathcal{V}}(\cdot), \pi_{\mathcal{V}}(\cdot))$. However, in general, $\nabla^{\mathfrak{B}}$ does not preserve the Riemannian metric $g$.
	
The torsion and curvature operator of ${\nabla}^{\mathfrak{B}}$ are defined respectively by
\begin{equation}
\begin{aligned}
T(X,Y)&={\nabla}^{\mathfrak{B}}_XY-{\nabla}^{\mathfrak{B}}_YX-[X,Y],\\
R(X,Y)Z&={\nabla}^{\mathfrak{B}}_X{\nabla}^{\mathfrak{B}}_YZ-{\nabla}^{\mathfrak{B}}_Y{\nabla}^{\mathfrak{B}}_XZ-{\nabla}^{\mathfrak{B}}_{[X,Y]}Z,
\end{aligned}
\end{equation}
and set 
\begin{equation} 
R(X,Y,Z,W)=\langle R(Z, W)X,Y \rangle
\end{equation}
for any $X,Y,Z,W\in \Gamma(TM)$, where $\langle \cdot, \cdot \rangle=g(\cdot, \cdot)$.
\begin{lemma}\label{tors}
For any $X,Y\in \Gamma(TM)$, we have
\begin{align}
T(X,Y)=-\pi_{\mathcal{V}}([\pi_{\mathcal{H}}(X), \pi_{\mathcal{H}}(Y)]).
\end{align}
\end{lemma}
\begin{proof}
This lemma immediately follows from \cite[Lemma 4.1]{ref27} and the integrability of the vertical distribution $\mathcal{V}$ ( i.e., $[\mathcal{V},\mathcal{V}]\subset \mathcal{V}$).
\end{proof}

The generalized Bott connection ${\nabla}^{\mathfrak{B}}$ on $M$ can induces a transverse connection $\nabla^T$ on $\mathcal{Q}$ by the isomorphic map $\sigma$, which is called the transverse Levi-Civita connection. More precisely, $\nabla^T$ is given by
\begin{align}
\nabla^T_Xs=\sigma^{-1}(\nabla^{\mathfrak{B}}_X\sigma(s))\label{2.6}
\end{align}
for any $X\in \Gamma(TM), s\in \Gamma(Q)$. Moreover, it is metric and torsion-free, i.e.,
\begin{equation}
\begin{aligned}
X\langle Z, W\rangle=\langle \nabla^T_X Z, W \rangle+\langle Z, \nabla^T_X W \rangle,\\
\nabla^T_X\pi_{\mathcal{H}}(Y)-\nabla^T_Y\pi_{\mathcal{H}}(X)-\pi_{\mathcal{H}}[X,Y]=0
\end{aligned}
\end{equation}
for any $X,Y\in \Gamma(TM)$ and $Z,W\in \Gamma(\mathcal{H})$. Note that here we have used the identification between $\mathcal{H}$ and $\mathcal{Q}$, and for simplicity, we will use this identification throughout this paper. The transverse curvature operator of $\nabla^T$ is defined by
\begin{equation}\begin{aligned} \label{trcu}
R^T(X,Y)Z&={\nabla}^{T}_X {\nabla}^{T}_YZ- {\nabla}^{T}_Y{\nabla}^{T}_XZ- {\nabla}^{T}_{[X,Y]}Z.
\end{aligned} \end{equation}
By the definitions of $R$ and $R^T$, it is easy to see that
\begin{align}
R^T(X,Y)Z=R(X,Y)Z\label{2.9.}
\end{align}
for any $X,Y\in \Gamma(TM)$ and $Z \in \Gamma(\mathcal{H})$. Combining \eqref{2.9.} and \cite[Corollary 5.12, p. 45]{ref2}, we obtian
\begin{lemma}\label{cure}
For any $X,Y\in \Gamma(TM)$ and $Z \in \Gamma(\mathcal{H})$, we have 
\begin{align}
R(X,Y)Z=R^T(X,Y)Z=R^T(\pi_\mathcal{H}(X),\pi_\mathcal{H}(Y))Z=R(\pi_\mathcal{H}(X),\pi_\mathcal{H}(Y))Z.\label{2.10.}
\end{align}
\end{lemma}

For convenience, on a Riemannian manifold $(M^m,g)$ with a Riemannian foliation $\mathcal{F}$, we shall use the following convention on the range of indices unless otherwise stated:
\begin{equation} 
\begin{aligned} 
	&A,B,C,\cdots= 1, 2,\dots, m;\\
	&\alpha,\beta,\gamma,\cdots = 1,\dots,q;\\
	&i,j,k,\dots = q+1,q+2, \dots, m,
\end{aligned} 
\end{equation}
where $q$ is the codimension of $\mathcal{F}$, and we agree that repeated indices are summed over the respective ranges.

Let  $\{e_A\}=\{e_\alpha,e_i\}$ be a local orthonormal frame field of $TM$, where $e_i \in \Gamma(\mathcal{V}), e_\alpha \in \Gamma(\mathcal{H})$,  and let $\{\omega^A\}=\{\omega^\alpha,\omega^i\}$ be its dual frame field. Let $R^A_{BCD}$ and ${(R^T)}^{{\alpha}}_{{\beta}{A}{B}}$ be the components of $R$ and $R^T$ with respect to $\{e_A\}$ respectively, that is,
\begin{align}
&R(e_A,e_B)e_C=R^D_{CAB} e_D,\\
&R^T(e_A,e_B)e_\beta={(R^T)}^{{\alpha}}_{{\beta}{A}{B}} e_\alpha.
\end{align}
From \eqref{2.9.}, \eqref{2.10.}, we know that
\begin{align}
&{(R^T)}^{{\alpha}}_{{\beta}{A}{i}}={R}^{{\alpha}}_{{\beta}{A}{i}}=0,\\
&{(R^T)}^{{\alpha}}_{{\beta}{j}{B}}={R}^{{\alpha}}_{{\beta}{j}{B}}=0,\\
&{(R^T)}^{{\alpha}}_{{\beta}{\gamma}{\delta}}={R}^{{\alpha}}_{{\beta}{\gamma}{\delta}}.\label{2.17}
\end{align}
Set 
\begin{align}
R^T_{{\beta}{\alpha}{\gamma}{\delta}}=R^T(e_{\beta}, e_{\alpha}, e_{\gamma}, e_{\delta})={(R^T)}^{{\alpha}}_{{\beta}{\gamma}{\delta}},
\end{align}
then it satisfies the following properties:
\begin{equation} 
\begin{aligned} \label{bich} 
&R^T_{{\alpha}{\beta}{\gamma}{\delta}} = -R^T_{{\alpha}{\beta}{\delta}{\gamma}}=-R^T_{{\beta}{\alpha}{\gamma}{\delta}},\\ 
&{R}^{T}_{{\alpha}{\beta}{\gamma}{\delta}} +{R}^{T}_{{\alpha}{\gamma}{\delta}{\beta}}+{R}^{T}_{{\alpha}{\delta}{\beta}{\gamma}}=0,\\
&R^T_{{\alpha}{\beta}{\gamma}{\delta}}=R^T_{{\gamma}{\delta}{\alpha}{\beta}}.
\end{aligned}
\end{equation}	
Let $Z=\sum_\alpha \xi^\alpha e_\alpha$ and $W=\sum_\beta \eta^\beta e_\beta$ be two horizontal vector fields on $M$, the transverse sectional curvature is defined by
\begin{equation}
\begin{aligned}
{K}^T(Z,W) &=\frac{ {R}^T(Z,W,W,Z)}{|Z|^2|W|^2 -\langle Z,W\rangle ^2}\\
	&=\frac {R^{T}_{\alpha\beta \gamma \delta} \xi^\alpha \eta^\beta \eta^\gamma \xi^\delta}{\sum_{\alpha}  (\xi^\alpha)^2\sum_{\alpha}  (\eta^\alpha)^2-(\sum_{\alpha}  \xi^\alpha\eta^\alpha)^2},
\end{aligned}
\end{equation}
and the transverse Ricci tensor is given by
\begin{align}
Ric^T(Z,W)=\langle R^T(Z, e_\alpha)e_\alpha, W \rangle,
\end{align}
so
\begin{align}
Ric^T_{\alpha\beta}=Ric^T(e_\alpha, e_\beta)=R^T_{\gamma\beta\alpha\gamma}. 
\end{align}  

On a Riemannian manifold $(M, g)$ with a Riemannian foliation $\mathcal{F}$, there is an important degenerate elliptic operator, namely, the sub-Laplacian operator. For a smooth function $u$ on $M$, the sub-Laplacian of $u$ is defined as
\begin{equation} \label{dela} 
\begin{aligned}
\Delta_\mathcal{H} u&=(\nabla^{\mathfrak{B}} du)(e_\alpha, e_\alpha) \\
	&=e_{\alpha}^2(u)-\left({\nabla}^{\mathfrak{B}}_{e_{\alpha}}e_{\alpha}\right)u,\\
	&=e_{\alpha}^2(u)-\left({\nabla}^{T}_{e_{\alpha}}e_{\alpha}\right)u.
\end{aligned}
 \end{equation}

Next, we will discuss the relationship between transverse curvature tensor $R^T$ and Riemannian curvature tensor $R^M$ on a Riemannian manifold $(M, g)$ with a foliation $\mathcal{F}$. Following the Nakagawa and Takagi's paper \cite{ref13}, one can define two tensors $A$ and $h$ of type (1,2) for $X, Y\in \Gamma(TM)$ by
 \begin{equation}
\begin{aligned} \label{2.21}
A(X,Y) & =-\pi_{\mathcal{V}}(\nabla^M_{\pi_{\mathcal{H}}(Y)} \pi_{\mathcal{H}}(X)), \\  
h(X,Y) & =\pi_{\mathcal{H}}(\nabla^M_{\pi_{\mathcal{V}}(Y)}\pi_{\mathcal{V}}(X)),
\end{aligned}
\end{equation}
where $\nabla^M$ is the Riemannian connection on $(M, g)$. Let $h_{ij} ^\alpha$ (resp. $A_{\alpha\beta} ^i$) be the components of $h$ (resp. $A$) with respect to the local frame field $\{e_A\}$. The foliation $\mathcal{F}$ is Riemannian if and only if 
\begin{align}
A_{\alpha\beta}^i=-A_{\beta\alpha}^i.\label{2.24.}
\end{align}
	
In \cite{ref13} or \cite{[Li]}, the authors gave the following formulas, which will be used later,
\begin{equation}  \label{har}
	\begin{aligned}
	&h_{i\beta; j} ^\alpha = h_{\beta i; j} ^\alpha = h_{i k} ^\alpha h_{k j} ^\alpha, \\
	&h^\alpha_{i\beta;\gamma}=A^\alpha_{\beta i;\gamma}=h^\alpha_{ik}A^k_{\beta\gamma},\\
	&A_{\alpha j; \beta}^i =-A_{\alpha \gamma}^i A_{\gamma \beta}^i, \\
	&h_{ij;k} ^\alpha - h_{ik;j} ^\alpha = R^M_{\alpha ijk} , \\
	&h_{ij;\beta} ^\alpha - h_{i\beta; j} ^\alpha + A_{\alpha j; \beta}^i - A_{\alpha \beta; j}^i =R^M_{\alpha ij\beta} ,\\
	&h_{i \beta; \gamma} ^\alpha - h_{i \gamma; \beta } ^\alpha + A_{\alpha \beta; \gamma}^i - A_{\alpha \gamma; \beta }^i = R^M_{\alpha i\beta \gamma},
	\end{aligned}
\end{equation}
where $h^A_{BC;D}$ (resp. $A^A_{BC;D}$) are the covariant derivatives of $h$ (resp. $A$) with respect to $\nabla^M$. Moreover, the Ricci identity for the second covariant derivative of $h$ with respect to $\nabla^M$ is given by
\begin{align}
h_{BC;DE}^A - h_{BC;ED}^A = h_{BC}^F R^M_{AFDE}+h_{FC}^A R^M_{BFDE} +h_{BF}^A R^M_{CFDE}.
\end{align}

Using the relationship between $\nabla^M$ and $\nabla^{\mathfrak{B}}$, we can deduce that
	
\begin{lemma}[cf. {\cite[Corollary 5.12]{ref2}}] \label{th}
		For any $X, Y, Z, W\in \mathcal{H}$, we have
		\begin{equation}
		\begin{aligned}
		\langle R^M(X,Y)Z,W \rangle=&\langle R^T(X,Y)Z,W \rangle +2\langle A(X,Y),A(Z,W) \rangle \\ &+\langle A(Y,W),A(X,Z)\rangle -\langle A(Y,Z),A(X,W) \rangle.
		\end{aligned}
		\end{equation}
\end{lemma}
	
\begin{corollary}\label{ric}
Let $(M^m,\mathcal{F}, g)$ be a Riemannian manifold with a Riemannian foliation $\mathcal{F}$, then
\begin{equation}
\langle R^M(X,e_\alpha)e_\alpha,Y \rangle =Ric^T (X,Y) + 3\langle A(X,e_\alpha),A(e_\alpha,Y) \rangle
\end{equation}
 for any $X, Y\in \mathcal{H}$.
\end{corollary}
	
\begin{proof}
According to Lemma \ref{th} and \eqref{2.24.}, we obtain
\begin{equation}
\begin{aligned}
\langle R^M(X,e_\alpha)e_\alpha,Y \rangle =&\langle R^T(X,e_\alpha)e_\alpha,Y \rangle+3\langle A(X,e_\alpha),A(e_\alpha,Y) \rangle \\
		& -\langle A(e_\alpha,e_\alpha),A(X,Y) \rangle \\
		=&Ric^T (X,Y) +3\langle A(X,e_\alpha),A(e_\alpha,Y) \rangle .
\end{aligned} 
\end{equation} 
\end{proof}
	
Finally, we will introduce some notions and notations of K\"ahler foliations.
\begin{definition}
A foliation $\mathcal{F}$ on a Riemannian manifold $(M,g)$ is said to be transversally holomorphic if there is a complex structure $J:\mathcal{Q} \to \mathcal{Q}$ on normal bundle $\mathcal{Q}$ satisfying $\mathcal{L}_VJ=0$ for any $V \in \Gamma(\mathcal{V})$. Furthermore, $\mathcal{F}$ is said to be K\"ahler if it is transversally holomorphic, and the complex structure $J$ satisfies
\begin{enumerate}\label{par}
\renewcommand{\labelenumi}{(\theenumi)}
\item $g_{\mathcal{Q}} (Jr,Js)=g_{\mathcal{Q}}(r,s)$, $ \forall r,s \in \Gamma ({\mathcal{Q}})$;
\item $\nabla^T J$=0.
\end{enumerate}
\end{definition}

Since $\nabla^T$ preserves the the metric $g_{\mathcal{Q}}$ and complex structure $J$, it is easy to see that the transverse curvature has the following properties.
\begin{proposition}
Let $(M^m, g, \mathcal{F},J)$ be a Riemannian manifold with a K{\"a}hler foliation, then
\begin{equation} 
\begin{aligned} 
R^T(X, Y, Z, W) &= R^T(JX, JY, Z, W) = R^T(X, Y, JZ, JW),\\
Ric^T(X, Y) &= Ric^T(JX, JY),
\end{aligned} 
\end{equation}
for $X, Y, Z, W \in \mathcal{H}$.
\end{proposition}
	
Let $(M^m, g)$ be a Riemannian manifold with a K\"ahler foliation $(\mathcal{F},J)$ of codimension $2q$. Denote the complexified horizontal distribution by $\mathcal{H}^{\mathbb{C}}=\mathcal{H}\otimes\mathbb{C}$. By extending $g$ and $J$ $\mathbb{C}$-linearly to $\mathcal{H}^{\mathbb{C}}$, the complexified horizontal distribution can be decomposed as 
\begin{align}
\mathcal{H}^{\mathbb{C}}=\mathcal{H}^{1,0}\oplus\mathcal{H}^{0,1},
\end{align}
where $\mathcal{H}^{1,0}$ and $\mathcal{H}^{0,1}$ are the eigenspaces of $J$ corresponding to eigenvalues $\sqrt{-1}$ and $-\sqrt{-1}$, respectively. Let $\{e_A\}=\{e_{\alpha},e_{\alpha^\ast},e_i\}$ be a local orthonormal frame field of $TM$, where $e_i \in \Gamma(\mathcal{V})$, $e_{\alpha}, e_{\alpha^\ast}\in \Gamma(\mathcal{H})$, and $e_{\alpha^\ast}=Je_{\alpha}$, and let $\{\omega^A\}=\{\omega^\alpha,\omega^{\alpha^\ast},\omega^i\}$ be its coframe field. Here we use the following convention on the ranges of indices:
\begin{equation} 
\begin{aligned} 
&A,B,C\dots =1,\dots,m;\\
&\alpha,\beta,\gamma\dots= 1,\dots,q;\\
&\alpha^\ast,\beta^\ast,\gamma^\ast\dots = q+1,\dots,2q;\\
&i,j,k\dots= 2q+1,\dots,m.
\end{aligned} 
\end{equation}	
Set $\eta_\alpha=\frac{\sqrt{2}}{2}(e_\alpha-\sqrt{-1}Je_\alpha)$, $\overline{\eta_{\alpha}}=\frac{\sqrt{2}}{2}(e_\alpha+\sqrt{-1}Je_\alpha)$, then $\{\eta_\alpha\}$ is a unitary frame field of $\mathcal{H}^{1,0}$. Let $(R^{T})^{\beta}_{\alpha \gamma \bar{\delta}}$ be the components of $R^T$ with respect to $\{\eta_\alpha, \overline{\eta_{\alpha}}\}$, that is
\begin{align}
R^T(\eta_\gamma,\overline{\eta_{\delta}})\eta_\alpha=(R^T)^\beta_{\alpha\gamma\bar{\delta}}\eta_\beta.
\end{align}
Set
\begin{align}
(R^T)^\beta_{\alpha\gamma\bar{\delta}}=R^T_{\alpha\bar{\beta}\gamma\bar{\delta}},
\end{align}
then it satisfies the following properties
\begin{align}
R^T_{\alpha\bar{\beta}\gamma\bar{\delta}}=-R^T_{\bar{\beta}\alpha\gamma\bar{\delta}}=-R^T_{\alpha\bar{\beta}\bar{\delta}\gamma},\ R^T_{\alpha\bar{\beta}\gamma\bar{\delta}}=R^T_{\gamma\bar{\beta}\alpha\bar{\delta}}=R^T_{\gamma\bar{\delta}\alpha\bar{\beta}}.
\end{align}
Suppose $Z=Z^\alpha \eta_\alpha$ and $W=W^\gamma \eta_\gamma$ are two vectors in $\mathcal{H}^{1,0}$, the transverse holomorphic bisectional curvature determined by $\{Z,W\}$ is defined by
\begin{equation}
\begin{aligned}
	H^T(Z,W) &=\frac{R^T(Z,\overline{Z},W,\overline{W})}{|Z|^2|W|^2}=\frac{ R^T_{\alpha \bar{\beta} \gamma \bar{\delta}}Z^\alpha \overline{Z^\beta} W^\gamma \overline{W^\delta} }{(Z^\alpha \overline{Z^\alpha})(W^\gamma \overline{W^\gamma})}.
\end{aligned}
\end{equation}
If $Z=W$, the above quantity is called the transverse holomorphic sectional curvature in the direction $Z$. The transverse Ricci curvature is defined as 
\begin{equation}
R^T_{\alpha \bar{\beta}} =R^T_{\gamma \bar{\gamma}\alpha \bar{\beta}}.
\end{equation}
	
\subsection{Transversally harmonic maps}
Let $(M^m,\mathcal{F},g)$ and $(N^n,\tilde{\mathcal{F}},\tilde{g})$ be two Riemannian manifolds with Riemannian foliations $\mathcal{F}$ and $\tilde{\mathcal{F}}$, respectively.
	
\begin{definition}\label{def2.7}
A smooth map $f: (M^m,\mathcal{F},g)\rightarrow (N^n,\tilde{\mathcal{F}},\tilde{g})$ is called a foliated map if $df(\mathcal{V}) \subset \tilde{\mathcal{V}}$, where $\mathcal{V}$ and $\tilde{\mathcal{V}}$ are the tangent bundles of $\mathcal{F}$ and $\tilde{\mathcal{F}}$, respectively. 
	\end{definition}
For a foliated map $f: (M^m,\mathcal{F},g)\rightarrow (N^n,\tilde{\mathcal{F}},\tilde{g})$, one can define the transverse second fundamental form ${\nabla}^{\mathfrak{B}}df_{{\mathcal{H}},\tilde{\mathcal{H}}}$  by 
\begin{equation} 
\begin{aligned}
{\nabla}^{\mathfrak{B}}df_{{\mathcal{H}},\tilde{\mathcal{H}}}(X,Y) &=({\nabla}^{\mathfrak{B}}_{X}df_{{\mathcal{H}},\tilde{\mathcal{H}}})(Y)   \\
&=\tilde{{\nabla}}^{\mathfrak{B}}_{df(X)}df_{{\mathcal{H}},\tilde{\mathcal{H}}}(Y)-df_{{\mathcal{H}},\tilde{\mathcal{H}}}({\nabla}^{\mathfrak{B}}_{X}{Y}),
\end{aligned}
\end{equation}
where $df_{{\mathcal{H}},\tilde{\mathcal{H}}} = \pi_{\tilde{\mathcal{H}}} \circ df \circ i_{\mathcal{H}},$ $\pi_{\tilde{\mathcal{H}}} : TN \to \tilde{\mathcal{H}}$ is the nature projection, $i_{\mathcal{H}} : \mathcal{H} \to TM$ is the inclusion morphism, $\mathcal{H}, \tilde{\mathcal{H}}$ are their horizontal distributions respectively.
	
\begin{definition}[cf. \cite{[KW]}] \label{def2.8}
A foliated map $f: (M^m,\mathcal{F},g)\rightarrow (N^n,\tilde{\mathcal{F}},\tilde{g})$ is called a transversally harmonic map if its transverse tension field $\tau_{{\mathcal{H}},\tilde{\mathcal{H}}}(f)=0$, where 
\begin{equation} 
\begin{aligned} \label{el} 
	\tau_{{\mathcal{H}},\tilde{\mathcal{H}}}(f)&= tr_{\mathcal{H}} ({\nabla}^{\mathfrak{B}}df_{{\mathcal{H}},\tilde{\mathcal{H}}})\\
	&=({\nabla}^{\mathfrak{B}}df_{{\mathcal{H}},\tilde{\mathcal{H}}})(e_{\alpha},e_{\alpha}) \\
	&=\tilde{\nabla}^{\mathfrak{B}}_{df(e_{\alpha})}df(e_{\alpha})-df_{{\mathcal{H}},\tilde{\mathcal{H}}}({\nabla}^{\mathfrak{B}}_{e_\alpha}{e_\alpha}).
\end{aligned}
\end{equation}  
\end{definition}

Let $\Omega$ be a compact subset of $M$, then the transversal energy of $f$ on $\Omega$ is given by
\begin{align}
E_{{\mathcal{H}}}(f;\Omega)= \int_{\Omega}e_{{\mathcal{H}}}(f)\dfrac{1}{vol_{\mathcal{F}}}\mu_M,
\end{align}
where $e_{{\mathcal{H}}}(f)= \dfrac{1}{2}{|df_{{\mathcal{H}},\tilde{\mathcal{H}}}|}^2=\sum_{\alpha} \langle df_{{\mathcal{H}},\tilde{\mathcal{H}}}(e_{\alpha}),df_{{\mathcal{H}},\tilde{\mathcal{H}}}(e_{\alpha}) \rangle$ is the transversal energy density of $f$, $vol_{\mathcal{F}}(x)$ is the volume of all leaves which pass through $x \in M$, $\mu_M$ is the volume element of $M$. 
\begin{proposition}[cf. \cite{ref3}]
A foliated map $f: (M^m,\mathcal{F},g)\rightarrow (N^n,\tilde{\mathcal{F}},\tilde{g})$ is transversally harmonic if and only if it is a critical point of the transversal energy $E_{{\mathcal{H}}}(f;\Omega)$ on any compact domain $\Omega\subset M$, when all leaves of $\mathcal{F}$ is compact.
\end{proposition} 

\begin{remark}
There is another different definition of transversally harmonic maps which was introduced by Barletta and Dragomir in \cite{ref6}, namely, a foliated map $f: (M^m,\mathcal{F},g)\rightarrow (N^n,\tilde{\mathcal{F}},\tilde{g})$ is called a transversally harmonic map if 
\begin{equation} \label{newde}
\tau_{{\mathcal{H}},\tilde{\mathcal{H}}}(f)-df_{{\mathcal{H}},\tilde{\mathcal{H}}}(\kappa)=0,
\end{equation} 
where $\kappa=\pi_{\mathcal{H}}(\nabla^M_{e_i}e_i)$ is the mean curvature of the foliation $\mathcal{F}$, $\nabla^M$ denotes the Riemannian connection on $M$. In particular, if the foliation $\mathcal{F}$ is harmonic, i.e., $\kappa=0$, these two definitions coincide. Furthermore, Barletta and Dragomir proved that a foliated map $f: (M^m,\mathcal{F},g)\rightarrow (N^n,\tilde{\mathcal{F}},\tilde{g})$ between two compact Riemannian manifolds with Riemannian foliations is transversally harmonic in the sense of \eqref{newde} if and only if it is a critical point of $E_T(f)=\frac{1}{2}\int_M|df_{{\mathcal{H}},\tilde{\mathcal{H}}}|^2\mu_M$.
\end{remark}
	
Let $(M^m,\mathcal{F},g,J)$  be a Riemannian manifold with K{\"a}hler foliation $\mathcal{F}$, and let $(N^n,\tilde{\mathcal{F}},\tilde{g},\tilde{J})$ be another Riemannian manifold with K{\"a}hler foliation $\tilde{\mathcal{F}}$.	
\begin{definition}\label{def2.11}
A foliated map $f:(M^m,\mathcal{F},g,J) \to  (N^n,\tilde{\mathcal{F}},\tilde{g},\tilde{J})$ is said to be transversally holomorphic if
\begin{align}
df_{{\mathcal{H}},\tilde{\mathcal{H}}} \circ J = \tilde{J} \circ df_{{\mathcal{H}},\tilde{\mathcal{H}}}.\label{2.39}
\end{align}
\end{definition}
\begin{remark}
According to \cite{ref4}, every transversally holomorphic map is transversally harmonic in the sense of Definition \ref{def2.8}.
\end{remark}
	
\subsection{Bochner-type formulas}	

Let $(M^m,\mathcal{F},g)$ be a Riemannian manifold with a Riemannian foliation $\mathcal{F}$. Under the orthonormal frame field $\{e_A\}=\{e_\alpha,e_i\}$ introduced in Section \ref{2.1}, the connection $1$-forms $\omega^A_B$ of generalized Bott connection ${\nabla}^{\mathfrak{B}}$ are given by
\begin{equation} \label{form}
	{\nabla}^{\mathfrak{B}}_{e_A} e_B=\omega^C_B(e_A)e_C,
	\end{equation}
where 
\begin{equation} \label{ojh}
	\omega^\alpha_i=0,\ \omega^j_\beta=0,
\end{equation}
because ${\nabla}^{\mathfrak{B}}$ preserves the decomposition \eqref{2.1}. According to \cite{[KN]}, as a linear connection, the structure equations of ${\nabla}^{\mathfrak{B}}$ are
\begin{equation} 
	\begin{cases}
	d\omega^A = -\omega^A_B \land \omega^B + T^A  &   \\ 
	d\omega^A_B = -\omega^A_C \land \omega^C_B + \Omega^A_B,  & \\  
	\end{cases} 
\end{equation}
with
\begin{equation}
T^A= \dfrac{1}{2}T^A_{CD} \omega^C \land \omega^D,
\end{equation} 
and
\begin{equation}
\Omega^A_B= \dfrac{1}{2}R^A_{BCD}\omega^C \land \omega^D, R^A_{BCD} = -R^B_{ACD},
\end{equation}
where $T^A_{CD}$ and $R^A_{BCD}$ are components of torsion $T$ and curvature tensor $R$ of ${\nabla}^{\mathfrak{B}}$, respectively. From Lemma \ref{tors} and Lemma \ref{cure}, it follows that
\begin{equation} 
\begin{aligned} \label{23}
T(\cdot,\cdot)=&T^ie_i=\dfrac{1}{2}(T^i_{\alpha \beta} \omega^\alpha \land \omega^\beta)\otimes e_i, \\
	&T^i_{\beta\gamma}=-T^i_{\gamma\beta},
\end{aligned}
\end{equation} 
and
\begin{equation} 
\begin{aligned} \label{24}
	&\Omega^\alpha_i=\Omega^j_\alpha=0, \\
	R^{\alpha}_{iAB}=R^{i}_{\alpha AB}&=0, \quad R^{A}_{\alpha iB} =R^{A}_{\alpha Cj}=0.
\end{aligned}
\end{equation} 

Suppose that $(N^n,\tilde{\mathcal{F}},\tilde{g})$ is another Riemannian manifold with a Riemannian foliation $\tilde{\mathcal{F}}$ and the generalized Bott connection $\tilde{\nabla}^{\mathfrak{B}}$. Let $\{\tilde{e}_{\tilde{A}}\}=\{\tilde{e}_{\tilde{i}}, \tilde{e}_{\tilde{\alpha}}\}$ be a local orthonormal frame field of $TN$ with $\tilde{e}_{\tilde{i}}\in \Gamma{(\tilde{\mathcal{V}})}$ and $\tilde{e}_{\tilde{\alpha}}\in \Gamma{(\tilde{\mathcal{H})}}$, and let $\{\tilde{\omega}^{\tilde{A}}\}=\{\tilde{\omega}^{\tilde{i}}, \tilde{\omega}^{\tilde{\alpha}}\}$ be its dual frame field. We will denote the corresponding geometric data, such as, the connection 1-forms, torsion and curvature, etc., on $N$ by the same notations as in $M$ , but with $\tilde{}$ on them. Then similar structure equations for $\tilde{\nabla}^{\mathfrak{B}}$ are also valid in $N$.

Let $f: (M^m,\mathcal{F},g)\rightarrow (N^n,\tilde{\mathcal{F}},\tilde{g})$ be a smooth map. In terms of the above frame fields, the differential $df:TM \to TN$ of $f$ can be written as 
\begin{equation}
df=  f^{\tilde{A}}_B \omega^B \otimes \tilde{e}_{\tilde{A}}.\label{2.47}
\end{equation}
From \eqref{2.47}, we have 
\begin{equation} \label{fenkai}
f^\ast \tilde{\omega}^{\tilde{A}}=f^{\tilde{A}}_B \omega^B.
\end{equation}

For simplification, we denote $\hat{\omega}^{\tilde{A}}_{\tilde{D}}=f^*\tilde{\omega}^{\tilde{A}}_{\tilde{D}}, {\hat{T}}^{\tilde{A}}_{{\tilde{C}}{\tilde{D}}}=f^*{\tilde{T}}^{\tilde{A}}_{{\tilde{C}}{\tilde{D}}}$ and so on. By taking the exterior derivative of (\ref{fenkai}) and using the structure equations in $M$ and $N$, we obtain 
\begin{equation} \label{15}
\mathbf{D}f^{\tilde{A}}_{B} \land \omega^{B} + \frac{1}{2}(f^{\tilde{A}}_DT^D_{BC}-f^{\tilde{C}}_Bf^{\tilde{D}}_C{\hat{T}}^{\tilde{A}}_{{\tilde{C}}{\tilde{D}}}) \omega^{B}\land\omega^C=0,
\end{equation}
where
\begin{equation} \label{16}
\mathbf{D}f^{\tilde{A}}_{B}= df^{\tilde{A}}_{B}-f^{\tilde{A}}_C \omega^C_B +f^{\tilde{D}}_B \hat{\omega}^{\tilde{A}}_{\tilde{D}}=f^{\tilde{A}}_{B,C} \omega^C,
\end{equation}
and $f^{\tilde{A}}_{B,C}$ are the covariant derivatives of $f^{\tilde{A}}_{B}$ with respect to the generalized Bott connection. From \eqref{15}, we deduce that 
\begin{equation}\label{chtwo}
f^{\tilde{A}}_{B,C}-f^{\tilde{A}}_{C,B}=f^{\tilde{A}}_D T^D_{BC}-f^{\tilde{C}}_Bf^{\tilde{D}}_C{\hat{T}}^{\tilde{A}}_{{\tilde{C}}{\tilde{D}}}.
\end{equation}
Applying the exterior derivative to \eqref{16} and using the structure equations again, we have
\begin{equation} \label{2.53}
\begin{aligned}
&\mathbf{D}f^{\tilde{A}}_{B,C} \land  \omega^C +\frac{1}{2}f^{\tilde{A}}_E R^E_{BCD}\omega^C\land\omega^D+\frac{1}{2}f^{\tilde{A}}_{B,E}T^E_{CD}\omega^C\land\omega^D\\
	&= \frac{1}{2}f^{\tilde{C}}_Bf^{\tilde{B}}_Df^{\tilde{E}}_C \hat{R}^{\tilde{A}}_{\tilde{C}\tilde{E}\tilde{B}}\omega^C\land\omega^D,
\end{aligned}
\end{equation}
where 
\begin{equation} 
\mathbf{D}f^{\tilde{A}}_{B,C}=df^{\tilde{A}}_{B,C} -f^{\tilde{A}}_{B,D} \omega^D_C -f^{\tilde{A}}_{D,C} \omega^D_B +f^{\tilde{C}}_{B,C} \hat{\omega}^{\tilde{A}}_{\tilde{C}}=f^{\tilde{A}}_{B,CD}\omega^D.
\end{equation} 
From \eqref{2.53}, it follows that
\begin{equation} \label{commu}
f^{\tilde{A}}_{B,CD} -f^{\tilde{A}}_{B,DC} = f^{\tilde{A}}_E R^E_{BCD}  - f^{\tilde{C}}_Bf^{\tilde{B}}_Df^{\tilde{E}}_C \hat{R}^{\tilde{A}}_{\tilde{C}\tilde{E}\tilde{B}}+f^{\tilde{A}}_{B,E}T^E_{CD}.
\end{equation}
	
Now we consider that $f: (M^m,\mathcal{F},g) \to (N^n,\tilde{\mathcal{F}},\tilde{g})$ is a foliated map. Clearly, the foliated condition ``$df(\mathcal{V}) \subset \tilde{\mathcal{V}}$" in Definition \ref{def2.7} is equivalent to 
\begin{align}
f^{\tilde{\alpha}}_i=0.\label{2.57}
\end{align}
Substituting \eqref{23} and \eqref{2.57} into \eqref{chtwo}, we get
\begin{equation}  \label{2.58}
f^{\tilde{\gamma}}_{\alpha, \beta}=f^{\tilde{\gamma}}_{\beta, \alpha},
\end{equation}
\begin{equation}
f^{\tilde{\gamma}}_{\alpha, i}=0.\label{2.59}
\end{equation}
From \eqref{23}, \eqref{24}, \eqref{commu}, \eqref{2.58}, \eqref{2.59}, it follows that
\begin{align}
f^{\tilde{\alpha}}_{\beta,\gamma\delta}-f^{\tilde{\alpha}}_{\beta,\delta\gamma}=f^{\tilde{\alpha}}_{\sigma}R^{\sigma}_{\beta\gamma\delta}-f^{\tilde{\beta}}_{\beta}f^{\tilde{\gamma}}_{\delta}f^{\tilde{\delta}}_{\gamma}\hat{R}^{\tilde{\alpha}}_{\tilde{\beta}\tilde{\delta}\tilde{\gamma}}\label{comh}.
\end{align}
	
The transversal energy density of $f$ is given by
\begin{equation}
e_{{\mathcal{H}}}(f)= \dfrac{1}{2}{|df_{{\mathcal{H}},\tilde{\mathcal{H}}}|}^2=\dfrac{1}{2}(f^{\tilde{\alpha}}_{\beta})^2.
\end{equation}
According to \eqref{dela}, we obtain 
\begin{equation}
\begin{aligned} \label{2.63}
\Delta_{\mathcal{H}} e_{{\mathcal{H}}}(f)&=\frac{1}{2}\Delta_{\mathcal{H}}({(f^{\tilde{\alpha}}_{\beta})}^2) \\
&=(f^{\tilde{\alpha}}_{\beta}f^{\tilde{\alpha}}_{\beta,\gamma})_\gamma \\
&=|f^{\tilde{\alpha}}_{\beta,\gamma}|^2 + f^{\tilde{\alpha}}_{\beta}f^{\tilde{\alpha}}_{\beta,\gamma\gamma}.
\end{aligned}
\end{equation}
Using \eqref{2.58} and \eqref{comh}, we perform the following computations
\begin{equation} 
\begin{aligned} \label{2.64}
f^{\tilde{\alpha}}_{\beta,\gamma\gamma} & = f^{\tilde{\alpha}}_{\gamma,\beta\gamma}\\
& = f^{\tilde{\alpha}}_{\gamma,\gamma\beta}+f^{\tilde{\alpha}}_\delta R^\delta_{\gamma\beta\gamma} - f^{\tilde{\beta}}_\gamma f^{\tilde{\delta}}_\beta f^{\tilde{\gamma}}_\gamma \hat{R}^{\tilde{\alpha}}_{\tilde{\beta}\tilde{\delta}\tilde{\gamma}}.\\
\end{aligned} 
\end{equation}
From \eqref{2.17}, \eqref{2.63}, \eqref{2.64}, it follows that
\begin{align}\label{2.65}
\Delta_{\mathcal{H}} e_{{\mathcal{H}}}(f)=|f^{\tilde{\alpha}}_{\beta,\gamma}|^2+f^{\tilde{\alpha}}_{\beta} f^{\tilde{\alpha}}_{\gamma,\gamma\beta}+f^{\tilde{\alpha}}_{\beta}f^{\tilde{\alpha}}_\delta Ric^T_{\beta\delta}-f^{\tilde{\alpha}}_{\beta} f^{\tilde{\beta}}_\gamma f^{\tilde{\varepsilon}}_\beta f^{\tilde{\gamma}}_\gamma \hat{R}^{T}_{\tilde{\beta}\tilde{\alpha}\tilde{\varepsilon}\tilde{\gamma}}.
\end{align}
Therefore we get the following 
\begin{lemma} \label{bochner}
Let $f: (M^m,\mathcal{F},g) \to (N^n,\tilde{\mathcal{F}},\tilde{g})$ be a transversally harmonic map. Then 
\begin{equation} \label{2.65.}
\Delta_{\mathcal{H}} e_{{\mathcal{H}}}(f)=|f^{\tilde{\alpha}}_{\beta,\gamma}|^2+f^{\tilde{\alpha}}_{\beta}f^{\tilde{\alpha}}_\delta Ric^T_{\beta\delta}-f^{\tilde{\alpha}}_{\beta} f^{\tilde{\beta}}_\gamma f^{\tilde{\varepsilon}}_\beta f^{\tilde{\gamma}}_\gamma \hat{R}^{T}_{\tilde{\beta}\tilde{\alpha}\tilde{\varepsilon}\tilde{\gamma}}. 
\end{equation}
\end{lemma}
\begin{proof}
Since $f$ is a transversally harmonic map, that is
\begin{equation} \label{2.67}
\tau_{{\mathcal{H}},\tilde{\mathcal{H}}}(f)=f^{\tilde{\alpha}}_{\gamma,\gamma}=0,
\end{equation}
the lemma immediately follows from \eqref{2.65} and \eqref{2.67}.
\end{proof}

Assume $f:(M^m,\mathcal{F},g,J) \to (N^n,\tilde{\mathcal{F}},\tilde{g},\tilde{J})$ is a transversally holomorphic map between two Riemannian manifolds with K{\"a}hler foliations. Under the local orthonormal frames $\{e_{\alpha},e_{\alpha^\ast},e_i\}$ of $TM$ and $\{\tilde{e}_{\tilde{\alpha}}, \tilde{e}_{\tilde{\alpha}^{\ast}},\tilde{e}_{\tilde{i}}\}$ of $TN$, which were introduced in Section 2.1, it is clear that \eqref{2.39} is equivalent to
\begin{equation}
f^{\tilde{\alpha}^{\ast}}_{\beta^{\ast}}= f^{\tilde{\alpha}}_{\beta},\ f^{\tilde{\alpha}}_{\beta^\ast}=- f^{\tilde{\alpha}^{\ast}}_{\beta},\label{2.67.}
\end{equation}
and the ``foliated" condition leads to
\begin{equation}
f^{\tilde{\alpha}}_i=0,\  f^{\tilde{\alpha}^{\ast}}_i=0.
\end{equation}

Since every transversally holomorphic map is transversally harmonic, then from Lemma \ref{bochner} we may deduce the Bochner formula for transversally holomorphic maps as follows.
\begin{lemma} \label{lemma2.14} 
Let $f:(M^m,\mathcal{F},g,J) \to (N^n,\tilde{\mathcal{F}},\tilde{g},\tilde{J})$ be a transversally holomorphic map, then we have 
\begin{equation} \label{bohoml} 
 \frac{1}{2}\Delta_{\mathcal{H}} e_{{\mathcal{H}}}(f)= |a^{\tilde{\alpha}}_{\beta\gamma}|^2+\overline{a^{\tilde{\alpha}}_{\beta}}a^{\tilde{\alpha}}_{\gamma}Ric^T_{\beta\bar{\gamma}}-a^{\tilde{\alpha}}_{\beta}\overline{a^{\tilde{\beta}}_{\beta}}a^{\tilde{\gamma}}_{\gamma}\overline{a^{\tilde{\delta}}_{\gamma}}\hat{R}^T_{\tilde{\alpha}\bar{\tilde{\beta}}\tilde{\gamma}\bar{\tilde{\delta}}},
\end{equation}
where $a^{\tilde{\alpha}}_{\gamma}$ and $a^{\tilde{\alpha}}_{\beta\gamma}$ are the components of $df_{\mathcal{H},\bar{\mathcal{H}}}$ and $\nabla^Tdf_{\mathcal{H}, \tilde{\mathcal{H}}}$ with respect to the unitary frame field $\{\eta_\alpha\}$ of $\mathcal{H}^{1,0}$.
\end{lemma}
\begin{proof}
Using the local orthonormal frames $\{e_{\alpha},e_{\alpha^\ast},e_i\}$ of $TM$ and $\{\tilde{e}_{\tilde{\alpha}}, \tilde{e}_{\tilde{\alpha}^{\ast}},\tilde{e}_{\tilde{i}}\}$ of $TN$, the Bochner formula \eqref{2.65.} can be rewritten as
\begin{equation}
\begin{aligned} \label{bochc}
\frac{1}{2}\Delta_{\mathcal{H}} e_{{\mathcal{H}}}(f)=&2(f^{\tilde{\alpha}}_{\beta,\gamma})^2 +2(f^{\tilde{\alpha}^{\ast}}_{\beta,\gamma})^2 +f^{\tilde{\alpha}}_{\beta}f^{\tilde{\alpha}}_{\delta_1} Ric^T_{\beta\delta_1}+f^{\tilde{\alpha}^{\ast}}_{\beta}f^{\tilde{\alpha}^{\ast}}_{\delta_1} Ric^T_{\beta\delta_1}\\
&-f^{\tilde{\alpha}}_{\beta}f^{\tilde{\beta_1}}_{\gamma_1} f^{\tilde{\delta_1}}_{\gamma_1} f^{\tilde{\gamma_1}}_{\beta}  \hat{R}^T_{\tilde{\beta_1}\tilde{\alpha}\tilde{\gamma_1}\tilde{\delta_1}}-f^{\tilde{\alpha}^\ast}_{\beta}f^{\tilde{\beta_1}}_{\gamma_1} f^{\tilde{\delta_1}}_{\gamma_1} f^{\tilde{\gamma_1}}_{\beta}  \hat{R}^T_{\tilde{\beta_1}\tilde{\alpha}^\ast\tilde{\gamma_1}\tilde{\delta_1}},
\end{aligned}
\end{equation}
where we have used \eqref{2.58} and \eqref{2.67.}, and add the following conventions on the ranges of indices in this proof:
\begin{equation}
\begin{aligned}
&{\alpha_1},{\beta_1},{\gamma_1}, \cdots = 1,\dots,2q;\\ 
&\tilde{\alpha_1},\tilde{\beta_1},\tilde{\gamma_1},\cdots = 1,\dots, 2\tilde{q};\\
&2q=\text{codim}\ \mathcal{F}, \ 2\tilde{q}=\text{codim}\ \tilde{\mathcal{F}}.
\end{aligned}
\end{equation}		
Now we will derive the Bochner formula for $e_{\mathcal{H}}$ with respect to the unitary frame field $\{\eta_\alpha\}$ of $\mathcal{H}^{1,0}$, where $\eta_\alpha=\frac{\sqrt{2}}{2}(e_\alpha-\sqrt{-1}e_{\alpha^*})$. Using $e_\alpha=\frac{1}{\sqrt{2}}(\eta_\alpha+\overline{\eta_{\alpha}})$ and $e_{\alpha^*}=\frac{\sqrt{-1}}{\sqrt{2}}(\eta_\alpha-\overline{\eta_{\alpha}})$, we obtain
 \begin{align}
 2(f^{\tilde{\alpha}}_{\beta,\gamma})^2 +2(f^{\tilde{\alpha}^{\ast}}_{\beta,\gamma})^2=2|(\nabla^{\mathfrak{B}}_{e_\gamma}df_{\mathcal{H}, \tilde{\mathcal{H}}})(e_\beta)|^2&=\frac{1}{2}|(\nabla^{\mathfrak{B}}_{\eta_\gamma+\overline{\eta_{\gamma}}}df_{\mathcal{H}, \tilde{\mathcal{H}}})(\eta_\beta+\overline{\eta_{\beta}})|^2\notag\\
 &=|a^{\tilde{\alpha}}_{\beta\gamma}|^2,\label{2.73}
 \end{align}
\begin{align} 
 &f^{\tilde{\alpha}}_{\beta} f^{\tilde{\alpha}}_{{\delta_1}} Ric^T_{\beta\delta_1}+  f^{\tilde{\alpha}^\ast}_{\beta}f^{\tilde{\alpha}^\ast}_{{\delta_1}} Ric^T_{\beta\delta_1}=\langle df_{{\mathcal{H}},\tilde{\mathcal{H}}}(Ric^T(e_{\beta})), df_{{\mathcal{H}}, \tilde{\mathcal{H}}}(e_{\beta}) \rangle \notag \\ 
 &= \langle df_{{\mathcal{H}},\tilde{\mathcal{H}}} (R^T(e_{\beta},e_\alpha)e_\alpha), df_{{\mathcal{H}}, \tilde{\mathcal{H}}}(e_{\beta}) \rangle  +\langle df_{{\mathcal{H}},\tilde{\mathcal{H}}}(R^T(e_{\beta},e_{\alpha^\ast})e_{\alpha^\ast}), df_{{\mathcal{H}}, \tilde{\mathcal{H}}}(e_{\beta}) \rangle\notag\\
 &=\langle df_{{\mathcal{H}},\tilde{\mathcal{H}}} (R^T(\overline{\eta_\beta},\eta_{\alpha})\overline{\eta_\alpha}), df_{{\mathcal{H}}, \tilde{\mathcal{H}}}(\eta_\beta)\rangle=\overline{a^{\tilde{\alpha}}_{\beta}}a^{\tilde{\alpha}}_{\gamma}Ric^T_{\beta\bar{\gamma}},\label{ricr1}
\end{align}
and
\begin{equation} \label{rr}
\begin{aligned}  
&f^{\tilde{\alpha}}_{\beta}f^{\tilde{\beta_1}}_{\gamma_1} f^{\tilde{\delta_1}}_{\gamma_1} f^{\tilde{\gamma_1}}_{\beta}  \hat{R}^T_{\tilde{\beta_1}\tilde{\alpha}\tilde{\gamma_1}\tilde{\delta_1}}+f^{\tilde{\alpha}^\ast}_{\beta}f^{\tilde{\beta_1}}_{\gamma_1} f^{\tilde{\delta_1}}_{\gamma_1} f^{\tilde{\gamma_1}}_{\beta}  \hat{R}^T_{\tilde{\beta_1}\tilde{\alpha}^\ast\tilde{\gamma_1}\tilde{\delta_1}}\\
=&\langle \tilde{R}^T(df_{{\mathcal{H}},\tilde{\mathcal{H}}}(e_{\beta}),df_{{\mathcal{H}},\tilde{\mathcal{H}}}(e_{\gamma}))df_{{\mathcal{H}},\tilde{\mathcal{H}}}(e_{\gamma}), df_{{\mathcal{H}}, \tilde{\mathcal{H}}}(e_{\beta}) \rangle \\&+\langle \tilde{R}^T(df_{{\mathcal{H}},\tilde{\mathcal{H}}}(e_{\beta}),\tilde{J} \circ df_{{\mathcal{H}},\tilde{\mathcal{H}}}(e_{\gamma}))\tilde{J} \circ df_{{\mathcal{H}},\tilde{\mathcal{H}}}(e_{\gamma}),df_{{\mathcal{H}}, \tilde{\mathcal{H}}}(e_{\beta}) \rangle \\
=&\langle \tilde{R}^T(df_{{\mathcal{H}},\tilde{\mathcal{H}}}(e_{\beta}),\tilde{J} \circ df_{{\mathcal{H}},\tilde{\mathcal{H}}}(e_{\beta}))\tilde{J} \circ df_{{\mathcal{H}},\tilde{\mathcal{H}}}(e_{\gamma}), df_{{\mathcal{H}}, \tilde{\mathcal{H}}}(e_{\gamma}) \rangle\\
=&\tilde{R}^T({df_{{\mathcal{H}},\tilde{\mathcal{H}}}(\eta_\beta)},{df_{{\mathcal{H}},\tilde{\mathcal{H}}}(\overline{\eta_{\beta}})},{df_{{\mathcal{H}},\tilde{\mathcal{H}}}(\eta_\gamma)},{df_{{\mathcal{H}},\tilde{\mathcal{H}}}(\overline{\eta_{\gamma}})})=a^{\tilde{\alpha}}_{\beta}\overline{a^{\tilde{\beta}}_{\beta}}a^{\tilde{\gamma}}_{\gamma}\overline{a^{\tilde{\delta}}_{\gamma}}\hat{R}^T_{\tilde{\alpha}\bar{\tilde{\beta}}\tilde{\gamma}\bar{\tilde{\delta}}}.
\end{aligned}
\end{equation}
where the second equality in \eqref{rr} follows from the Bianchi identity for $\tilde{R}^T$. From \eqref{bochc} and \eqref{2.73}-\eqref{rr}, the Bochner formula \eqref{bohoml} follows.
\end{proof}

\section{The sub-Laplacian comparison theorem} \label{sec3}
In this section, we will establish a comparison theorem for sub-Laplacian operator on a Riemannian manifold with a Riemannian foliation, which plays a similar role as the Laplacian comparison theorem in Riemannian geometry. 

\begin{theorem}\label{lapp} 
Let  $(M^m,\mathcal{F},g)$ be a complete Riemannian manifold with a Riemannian foliation $\mathcal{F}$. Assume that
\begin{align}
Ric^T(\nabla^{\mathcal{H}} r, \nabla^{\mathcal{H}} r)\geq -K_1 |\nabla^{\mathcal{H}} r|^2,
\end{align}
and 
\begin{align}
|A|_{C_1}=\sup_{y \in M}\{|A|(y),|\nabla A|(y)\} \le k_1,\\
|h|_{C_1}=\sup_{y \in M}\{|h|(y),|\nabla h|(y)\} \le k_2,
\end{align}
where $\nabla^{\mathcal{H}}=\pi_{\mathcal{H}}\circ\nabla$ is the horizontal gradient, $K_1, k_1, k_2\geq 0$, and $r$ is the Riemannian distance from a fixed point $x_0$ in $M$. 	
Then if $x$ is not in the cut locus of $x_0$, we have
\begin{equation}
\Delta_\mathcal{H} r(x) \leq q\left(\frac{1}{r}+\sqrt{C(p,q)\left(K_1+k_1+k_2+k_1k_2+k_1^2+k_2^2\right)}\right),
\end{equation}
where $q$ (resp. $p$) is the codimension (resp. dimension) of the foliation $\mathcal{F}$, and $C(p,q)$ is a constant depending only on $p$ and $q$.
\end{theorem}
\begin{proof}
Let $\gamma :[0,+\infty)\rightarrow M$ be a geodesic with respect to the Levi-Civita connection $\nabla^M$ with $\gamma(0)=x_0$ and $\gamma'(0)=e$, where $e$ is a uint vector in $T_{x_0}M$. Let $\{e_A\}=\{e_i,e_\alpha\}$ be a local orthonormal frame field of $TM$ in a small neighborhood of $\gamma$, where $e_i \in \Gamma(\mathcal{V}), e_\alpha \in \Gamma(\mathcal{H})$. Since $|\nabla r|=1$ holds outside the cut locus of $x_0$, taking covariant derivatives with respect to $\nabla^M$ and summing yield
\begin{equation} \label{rbu}
\begin{aligned}
0 = \frac{1}{2} |\nabla r|_{\beta;\beta}^2 = &(r_{i;\beta})^2 + (r_{\alpha;\beta})^2 +r_ir_{i;\beta\beta} +r_\alpha r_{\alpha;\beta\beta} \\
	\ge &  (r_{\alpha;\alpha})^2 + r_ir_{\beta; i\beta} +r_\alpha r_{\beta;\alpha\beta} \\
	\ge & \frac{1}{q} (\sum_{\alpha} r_{\alpha;\alpha})^2 +  r_i(r_{\beta;\beta i}+r_C R_{\beta C i \beta}^{M})\\
	&\ \ \ \ \ \ \ \ \ \ \ \ \ \ \ \  + r_\alpha(r_{\beta;\beta\alpha}+ r_CR_{\beta C \alpha \beta}^{M}),
\end{aligned}
\end{equation}	
where ``;" means taking covariant derivatives with respect to the Riemannian connection $\nabla^M$. By \eqref{bott}, we get
\begin{align}
\omega^\gamma_\beta(e_\alpha)={}^M\omega^\gamma_\beta(e_\alpha),\label{3.6}
\end{align}
where ${}^M\omega^\gamma_\beta$ denote the connection 1-forms of $\nabla^M$. Since $\mathcal{F}$ is a Riemannian foliation, namely, $A_{\alpha \beta}^i=-A_{\beta \alpha}^i$, we obtain
\begin{equation} \label{3.7}
{}^M\omega_{\beta}^{i}(e_\beta) = - A_{\beta \beta}^i=0.
\end{equation}
From \eqref{3.6} and \eqref{3.7}, it follows that
\begin{equation}
r_{\beta;\beta} =  e_\beta ^2 r -{}^M\omega_{\beta}^{\gamma}(e_\beta) r_\gamma - {}^M\omega_{\beta}^{i}(e_\beta) r_i=  e_\beta ^2 r - \omega_{\beta}^{\gamma}(e_\beta) r_\gamma=r_{\beta,\beta},
\end{equation}
where ``," means taking covariant derivatives with respect to the generalized Bott connection $\nabla^\frak{B}$. Hence, 
\begin{equation} \label{rh}
\Delta_\mathcal{H} r= \sum_{\alpha} r_{\alpha,\alpha}=\sum_{\alpha} r_{\alpha;\alpha}.
\end{equation}
Substituting (\ref{rh}) into (\ref{rbu}) yields
\begin{equation}
\begin{aligned} \label{buds}
	0\geq \frac{1}{q} (\Delta_\mathcal{H} r)^2 + \langle \nabla \Delta_\mathcal{H} r, \nabla r \rangle + 2r_\alpha r_i R_{\beta\alpha i \beta}^{M}+ r_ir_j R_{\beta j i \beta}^{M} + r_\alpha r_\gamma R_{\beta\gamma\alpha\beta}^{M} .
\end{aligned}
\end{equation}	
From \eqref{2.24.} and \eqref{har}, it follows that 
\begin{equation} 
\begin{aligned} \label{bds12}
	&r_\alpha r_i R_{\beta\alpha i \beta}^{M} =- r_\alpha r_i A_{\beta\beta \alpha}^i +r_\alpha r_i A_{\beta\alpha\beta}^i -2 r_\alpha r_i h_{ik}^\beta A_{\beta\alpha}^k, \\
	&r_i r_j R_{\beta j i \beta}^{M} = r_i r_j h_{ij\beta}^\beta -r_i r_j A_{\beta\beta j}^i - r_i r_j h_{ik}^\beta h_{kj}^\beta - r_i r_j A_{\beta \gamma}^i A_{\gamma \beta}^j.
\end{aligned}
\end{equation}
According to Corollary \ref{ric}, we have
\begin{equation} \label{bds3}
r_\alpha r_\gamma R_{\beta\gamma\alpha\beta}^{M} =r_\alpha r_\gamma R^T_{\beta\gamma\alpha\beta} - 3r_\beta  r_\gamma A_{\alpha \beta}^i A_{\alpha\gamma}^i.
\end{equation}
In terms of \eqref{buds}-\eqref{bds3}, we deduce that 
\begin{equation} \label{fjx}
\begin{aligned}
	0 &\ge \frac{1}{q} (\Delta_\mathcal{H} r)^2 + \langle \nabla \Delta_\mathcal{H} r, \nabla r \rangle -2( r_\alpha r_i A_{\beta\beta \alpha}^i -r_\alpha r_i A_{\beta\alpha\beta}^i +2r_\alpha r_i h_{ij}^\beta A_{\beta\alpha}^j) \\
	&\qquad -(-r_i r_j h_{ij\beta}^\beta +r_i r_j A_{\beta\beta j}^i +r_i r_j h_{ik}^\beta h_{kj}^\beta + r_i r_j A_{\beta \gamma}^i A_{\gamma \beta}^j)\\
	&\quad \qquad +r_\alpha r_\gamma R^T_{\beta\gamma\alpha\beta} -3 r_\beta  r_\gamma A_{\alpha \beta}^i A_{\alpha\gamma}^i.
\end{aligned}
\end{equation}
Since $Ric^T(\nabla^{\mathcal{H}} r, \nabla^{\mathcal{H}} r)\geq -K_1 |\nabla^{\mathcal{H}} r|^2,\ |A|_{C^1}\leq k_1, \ |h|_{C^1}\leq k_2$, and using the Cauchy-Schwarz inequality, it follows from \eqref{fjx} that
\begin{align}\label{wffc}
0 &\ge\frac{1}{q} (\Delta_\mathcal{H} r)^2 + \langle \nabla \Delta_\mathcal{H} r, \nabla r \rangle - q\cdot C(p,q) (K_1+k_1+k_2+k_1k_2+k_1^2+k_2^2),
\end{align}
where $C(p,q)$ is a constant depending only on $p$ and $q$.
Set 
\begin{align}
&f(t)=\Delta_\mathcal{H} r(\gamma(t)),\\
\tilde{C}=C(p,q)& (K_1+k_1+k_2+k_1k_2+k_1^2+k_2^2),
\end{align}
then \eqref{wffc} can be rewritten as
\begin{equation}  \label{fcz}
0 \ge \frac{1}{q} f^2(t)+f'(t)-q\tilde{C}.
\end{equation}
Since a Riemannian metric is locally Euclidean, we have
\begin{equation}
\lim_{t \to 0^+} tf(t) = q - |\pi_{\mathcal{H}}(e)|^2,
\end{equation}
which implies that  
\begin{equation}
\lim_{t \to 0^+} f(t) = +\infty.
\end{equation}
Hence, there exists a real number $T>0$ such that
\begin{equation}
f^2(T) - q^2\tilde{C} = 0,
\end{equation} 
\begin{equation} \label{tnd}
f^2(t) - q^2 \tilde{C} > 0, \ \forall\  t\in (0,T).
\end{equation}
Combining (\ref{fcz}) and (\ref{tnd}) yields 
\begin{equation}
\frac{qf'(t)}{f^2(t)- q^2\tilde{C}} \le -1,\ \forall\ t\in (0,T),
\end{equation}
which is equivalent to
\begin{equation} 
\begin{aligned}
\frac{d}{dt} coth^{-1} \left(\frac{f(t)}{q \sqrt{\tilde{C}}}\right) \ge  \sqrt{\tilde{C}},\ \forall\ t\in (0,T),
\end{aligned}		
\end{equation}
thus, 
\begin{equation} \label{bdc}
\begin{aligned}
f(t) &\le q \cdot  \sqrt{\tilde{C}} \cdot coth( \sqrt{\tilde{C}} \cdot t) \\
     &\le q(\frac{1}{t}+ \sqrt{\tilde{C}}),\ \forall\ t\in (0,T).
\end{aligned}
\end{equation}		
We claim that (\ref{bdc}) holds in $(0,+\infty)$. Otherwise, there is a real number $t_1 > T$ such that 
\begin{equation}
f(t_1)> q\sqrt{\tilde{C}}.
\end{equation}
Then we can find a constant $t_2 \in (T, t_1)$ satisfying 
\begin{equation}
f'(t_2) \ge 0,\ f(t_2) >q \sqrt{\tilde{C}},
\end{equation}
thus,
\begin{equation} \label{md}
\frac{1}{q} f^2(t_2)+f'(t_2)-q\tilde{C} >0,
\end{equation}
which is a contradiction with (\ref{fcz}). Therefore, we have 
 \begin{equation}
\Delta_\mathcal{H} r \le q(\frac{1}{r}+\sqrt{C(p,q) (K_1+k_1+k_2+k_1k_2+k_1^2+k_2^2)})
\end{equation}
outside the cut locus of $x_0$.
\end{proof}
	
\section{The Schwarz lemmas}
Let $(M^m,\mathcal{F},g)$ and $(N^n,\tilde{\mathcal{F}},\tilde{g})$ be two Riemannian manifolds with Riemannian foliations $\mathcal{F}$ and $\tilde{\mathcal{F}}$ respectively, and suppose that $f:M \to N$ is a foliated map.
Under the local frame fields introduced in Section \ref{sec2.1}, the ``horizontal component of $g$ and $\tilde{g}$” can be expressed as
\begin{equation}
g_{\mathcal{H}}=(\omega^{\alpha})^2,
\end{equation}
\begin{equation}
\tilde{g}_{\tilde{\mathcal{H}}}=(\tilde{\omega}^{\tilde{\alpha}})^2,
\end{equation}
respectively. If we define a matrix  $\mathbf{A}=(\mathbf{A}_{\alpha\beta})$, where 
\begin{equation}
\mathbf{A}_{\alpha\beta}=f^{\tilde{\alpha}}_{\alpha}f^{\tilde{\alpha}}_{\beta},
\end{equation} 
and let $\lambda$ be the maximal eigenvalue of $\mathbf{A}$, then we have 
\begin{equation}\label{lamb} 
f^*\tilde{g}_{\tilde{\mathcal{H}}}=\mathbf{A}_{\alpha\beta}\omega^\alpha\omega^\beta\le \lambda  g_{\mathcal{H}},
\end{equation}
since $f$ is foliated.
\begin{definition}\label{def4.1}
Let $\{\lambda_i\}_{i=1}^q$ be the eigenvalues of $\mathbf{A}$ with
$$\lambda(=\lambda_1)\ge\lambda_2\geq\dots\ge\lambda_q \ge 0.$$
If there exists a constant $\beta >0$ for the map $f:M \to N$ such that 
$$\lambda \le \beta^2(\lambda_2+\dots+\lambda_q),$$ 
then we call $f$ to be of bounded generalized transversal dilatation of order $\beta$.	
\end{definition}

\begin{proof}[\textbf{Proof of Theorem \ref{thm1.1}}]	
According to (\ref{lamb}), it suffices to show that $\lambda \le \beta^2 \frac{K_1}{K_2}$. For arbitrary point $z$ of $M$, we define the function 
\begin{equation}
\Phi(x)=(a^2-r^2(x))^2\lambda(x),
\end{equation} where $x \in \overline{B_a (z)}=\overline{\{x\in M,r(x)< a\}}$, and $r(x)$ is the Riemannian distance function from $z$ to $x$. Since $M$ is complete, and $\Phi\geq 0, \Phi|_{\partial B_a(z)}\equiv 0$, there exists a point $x_0\in B_a (z)$ such that $\Phi(x_0)=\max_{x\in \overline{B_a (z)}}\Phi(x)$. In order to estimate the upper bound of $\lambda$, we will apply the maximum principle to $\Phi$ at $x_0$, but $r(x)$ and $\lambda(x)$ may not be $C^2$ near the point  $x_0$. The non-differentiability of $r$ and $\lambda$ can be remedied by the following methods. For $r(x)$, we may assume that it is smooth near $x_0$, otherwise it can be modified as usual (cf. \cite{ref8}). For $\lambda(x)$, let $\eta(x_0)\in \mathcal{H}$ be the unit eigenvector of $\lambda(x_0)$, and we displace it parallelly with respect to $\nabla^\mathfrak{B}$ along any Riemannian geodesic started from $x_0$ to obtain a local vector field $\eta(x)=\eta^\alpha e_\alpha$ near $x_0$. Define the function 
\begin{equation}
	\tilde{\lambda}=|df_{{\mathcal{H}},\tilde{\mathcal{H}}}(\eta)|^2 = \sum_{\tilde{\alpha}} |f^{\tilde{\alpha}}_{\alpha}\eta^\alpha|^2,
\end{equation}which is smooth near $x_0$, and 
\begin{equation} \notag 
\begin{aligned}
	&\tilde{\lambda}(x) \le \lambda(x), \\
	&\tilde{\lambda}(x_0) = \lambda(x_0).
\end{aligned}
\end{equation}
Let
\begin{align}
\tilde{\Phi}(x)=(a^2-r^2(x))^2 \tilde{\lambda}(x),
\end{align}
then $\tilde{\Phi}(x)$ also attains its maximum at $x_0$ in $B_a(z)$ and $\tilde{\Phi}(x_0)={\Phi}(x_0)$. Without loss of generality, we suppose that $\lambda\not\equiv 0$ on $B_a(z)$, so $\tilde{\lambda}(x_0)=\lambda(x_0)>0$. Applying the maximum principle to $\tilde{\Phi}$ at $x_0$ yields
\begin{equation}
0=\frac{{\nabla}^\mathcal{H}\tilde{\lambda}}{\tilde{\lambda}} +2\frac{{\nabla}^\mathcal{H}(a^2-r^2)}{(a^2-r^2)},
\end{equation}
\begin{equation} \begin{aligned} \label{qd}
0\ge(a^2-&r^2)^2\Delta_{\mathcal{H}}\tilde{\lambda} -2\tilde{\lambda}(a^2-r^2)\Delta_{\mathcal{H}}r^2 \\
&-6\tilde{\lambda}|{\nabla}^\mathcal{H}(a^2-r^2)|^2.
\end{aligned}	
\end{equation}
By $|\nabla r|=1$, we have 
$$|{\nabla}^\mathcal{H}r|^2\le 1,$$ 
thus,
\begin{equation} \label{14} 
 |{\nabla}^\mathcal{H}(a^2-r^2)|^2=|2r{\nabla}^\mathcal{H}r|^2\le 4a^2. 
\end{equation}
According to Theorem \ref{lapp}, we obtain 
\begin{equation}
\begin{aligned} \label{12}
	\Delta_{\mathcal{H}}r^2&=2r\Delta_{\mathcal{H}}r+2|{\nabla}^\mathcal{H}r|^2 \\
	& \le 2aq(\frac{1}{a}+C)+2 \\
	&\le 2q(1+a\cdot C)+2,
\end{aligned} 
\end{equation}
where $C$ is a positive constant independent of $a$.
By definition, we have	
\begin{equation}
\tilde{\lambda}= f^{\tilde{\alpha}}_{\alpha}\eta^\alpha  f^{\tilde{\alpha}}_{\beta}\eta^\beta,
\end{equation} 
then
\begin{equation}
\tilde{\lambda}_\gamma= 2 (f^{\tilde{\alpha}}_{\alpha,\gamma}\eta^\alpha  f^{\tilde{\alpha}}_{\beta}\eta^\beta + f^{\tilde{\alpha }}_{\alpha}\eta^\alpha_{\ ,\gamma}  f^{\tilde{\alpha}}_{\beta}\eta^\beta ).
\end{equation}
Due to (\ref{dela}), we perform the following computation
\begin{equation} 
\begin{aligned} \label{clac}
\Delta_{\mathcal{H}}\tilde{\lambda} = &\tilde{\lambda}_{\gamma,\gamma} \\
	= & 2 ( f^{\tilde{\alpha}}_{\alpha, \gamma\gamma}\eta^\alpha  f^{\tilde{\alpha}}_{\beta}\eta^\beta + 2 f^{\tilde{\alpha}}_{\alpha,\gamma}\eta^\alpha_{\ ,\gamma}  f^{\tilde{\alpha}}_{\beta}\eta^\beta \\ 
	& +f^{\tilde{\alpha}}_{\alpha,\gamma}\eta^\alpha  f^{\tilde{\alpha}}_{\beta, \gamma}\eta^\beta  + 2 f^{\tilde{\alpha}}_{\alpha,\gamma}\eta^\alpha  f^{\tilde{\alpha}}_{\beta}\eta^\beta_{\ ,\alpha} \\ 
	& +f^{\tilde{\alpha}}_{\alpha}\eta^\alpha_{\ ,\gamma\gamma} f^{\tilde{\alpha}}_{\beta}\eta_\beta + f^{\tilde{\alpha}}_{\alpha  }\eta^\alpha_{\ ,\gamma} f^{\tilde{\alpha}}_{\beta}\eta^\beta_{\ ,\gamma}). 
\end{aligned}
\end{equation} 
Let $\{x^A\}=\{x^i, x^\alpha\}$ be a normal coordinate system with respect to Levi-Civita connection $\nabla^M$ at $x_0$ and satisfying $\frac{\partial}{\partial x^i}|_{x_0}\in \mathcal{V}, \frac{\partial}{\partial x^\alpha}|_{x_0}\in\mathcal{H}.$ By the parallel displacement of $\{\frac{\partial}{\partial x^A}|_{x_0}\}$ with respect to $\nabla^{\mathfrak{B}}$ along any Riemannian geodesic $\sigma(t) : [0,\infty)\rightarrow M$ which starts from $x_0$, we can obtain a local orthonormal frame $\{e_A\}=\{e_i,e_\alpha\}$ around $x_0$ satisfying $\text{span}\{e_i\}=\mathcal{V}$ and $\text{span}\{e_\alpha\}=\mathcal{H}$. Then we have 
\begin{align}
\omega^A=x^Adt\label{4.16}
\end{align}
along the geodesic $\sigma(t)$, where $\{\omega^A\}$ is the dual frame of $\{e_A\}$. Using the structure equations of $\nabla^{\mathfrak{B}}$, we obtain
\begin{align}
dx^\alpha+x^\beta\omega^\alpha_\beta=0.\label{4.17}
\end{align}
Since $\eta$ is parallel along $\sigma$, i.e., $\nabla^{\mathfrak{B}}_{\frac{\partial}{\partial t}}\eta=0$, then 
\begin{align}
\eta^\alpha_{\ ,\beta}\omega^\beta=0,\label{4.18}
\end{align}
so we get
\begin{align}
\eta^\alpha_{\ ,\beta}x^{\beta}=0
\end{align}
along the geodesic $\sigma$. Therefore, 
\begin{align}
\eta^\alpha_{\ ,\beta}(x_0)=0,\label{4.20}
\end{align}
because $\{x^A\}$ is an arbitrary unit tangent vector. Differentiating \eqref{4.20} with respect to $t$ and using \eqref{4.16} and \eqref{4.17} yield
\begin{equation}
\begin{aligned}
0&=d\eta^\alpha_{\ ,\beta}\cdot x^\beta+\eta^\alpha_{\ ,\beta}\cdot dx^\beta\\
&=(d\eta^\alpha_{\ ,\beta}-\eta^\alpha_{\ ,\gamma}\omega^\gamma_\beta+\eta^\gamma_{\ ,\beta}\omega^\alpha_\gamma-\eta^\gamma_{\ ,\beta}\omega^\alpha_\gamma)x^\beta\\
&=(\eta^\alpha_{\ ,\beta A}\omega^A-\eta^\gamma_{\ ,\beta}\omega^\alpha_\gamma)x^\beta\\
&=(\eta^\alpha_{\ ,\beta A}x^Adt-\eta^\gamma_{\ ,\beta}\omega^\alpha_\gamma)x^\beta.\label{4.21}
\end{aligned}
\end{equation}
From \eqref{4.20} and \eqref{4.21}, it follows that $\eta^\alpha_{\ ,\beta A}x^Ax^\beta=0$ at $x_0$. Since $\{x^A\}$ is an arbitrary unit tangent vector, then 
\begin{align}
\eta^\alpha_{\ ,\gamma \gamma}(x_0)=0.\label{4.22}
\end{align}
 Substituting \eqref{4.20} and \eqref{4.22} into \eqref{clac}, we have at $x_0$
	\begin{equation} \begin{aligned} \label{cla}
	\Delta_{\mathcal{H}}\tilde{\lambda} = & 2 ( f^{\tilde{\alpha}}_{\alpha, \gamma\gamma}\eta_\alpha  f^{\tilde{\alpha}}_{\beta}\eta_\beta + f^{\tilde{\alpha}}_{\alpha,\gamma }\eta_{\alpha} f^{\tilde{\alpha}}_{\beta,\gamma}\eta_{\beta} ) \\ 
\ge &  2 f^{\tilde{\alpha}}_{\alpha, \gamma\gamma}\eta_\alpha  f^{\tilde{\alpha}}_{\beta}\eta_\beta 
	\end{aligned}
	\end{equation} 
From \eqref{2.58} and \eqref{comh}, we obtain
\begin{equation} 
\begin{aligned}
f^{\tilde{\alpha}}_{\alpha, \gamma\gamma} &= f^{\tilde{\alpha}}_{\gamma, \alpha\gamma} \\ 
&=  f^{\tilde{\alpha}}_{\gamma, \gamma\alpha} + f^{\tilde{\alpha}}_\delta R^\delta_{\gamma \alpha \gamma} - f^{\tilde{\beta}}_\gamma f^{\tilde{\gamma}}_\alpha f^{\tilde{\delta}}_\gamma \tilde{R}^{\tilde{\alpha}}_{\tilde{\beta}\tilde{\gamma}\tilde{\delta}}.
\end{aligned}
\end{equation}
Thus,
\begin{equation}
\begin{aligned}
f^{\tilde{\alpha}}_{\alpha, \gamma\gamma}\eta_\alpha  f^{\tilde{\alpha}}_{\sigma}\eta_\sigma &= f^{\tilde{\alpha}}_\delta R^\delta_{\gamma \alpha \gamma} \eta_\alpha f^{\tilde{\alpha}}_{\sigma}\eta_\sigma- \hat{R}^{\tilde{\alpha}}_{\tilde{\beta}\tilde{\gamma}\tilde{\delta}}  f^{\tilde{\beta}}_\gamma \cdot f^{\tilde{\alpha}}_{\sigma}\eta_\sigma \cdot f^{\tilde{\delta}}_\gamma \cdot f^{\tilde{\gamma}}_\alpha \eta_\alpha \\ 
&= R^\delta_{\gamma \alpha \gamma} \eta_\alpha \mathbf{A}_{\delta \sigma} \eta_\sigma - \hat{R}^{\tilde{\alpha}}_{\tilde{\beta}\tilde{\gamma}\tilde{\delta}}  f^{\tilde{\beta}}_\gamma \cdot f^{\tilde{\alpha}}_{\sigma}\eta_\sigma \cdot f^{\tilde{\delta}}_\gamma \cdot f^{\tilde{\gamma}}_\alpha \eta_\alpha \\ 
&=R^\delta_{\gamma \alpha \gamma} \eta_\alpha \cdot \lambda \eta_\sigma - \hat{R}^{\tilde{\alpha}}_{\tilde{\beta}\tilde{\gamma}\tilde{\delta}}  f^{\tilde{\beta}}_\gamma \cdot f^{\tilde{\alpha}}_{\sigma}\eta_\sigma \cdot f^{\tilde{\delta}}_\gamma \cdot f^{\tilde{\gamma}}_\alpha \eta_\alpha.
\end{aligned}
\end{equation}
According to 
\begin{equation}
R^\delta_{\gamma \alpha \gamma} =Ric^T_{\alpha\delta}
\end{equation} 
and the curvature assumption on $M$, we have
\begin{equation} \label{ric1}
\begin{aligned}
R^\delta_{\gamma \alpha \gamma} \eta_\alpha \cdot \lambda \eta_\sigma &= (Ric^T_{\alpha\delta})\lambda \eta_\sigma \cdot \eta_\alpha\ge  -K_1\lambda.
\end{aligned} 
\end{equation}
Using the curvature assumption on $N$ yields	
\begin{equation} \label{sec}
\begin{aligned}
\hat{R}^{\tilde{\alpha}}_{\tilde{\beta}\tilde{\gamma}\tilde{\delta}}  f^{\tilde{\beta}}_\gamma \cdot f^{\tilde{\alpha}}_{\sigma}\eta_\sigma \cdot f^{\tilde{\delta}}_\gamma \cdot f^{\tilde{\gamma}}_\alpha \eta_\alpha	\le & -K_2[(f^{\tilde{\beta}}_\gamma f^{\tilde{\beta}}_\gamma) (f^{\tilde{\alpha}}_{\sigma}\eta_\sigma f^{\tilde{\alpha}}_\alpha \eta_\alpha)-(f^{\tilde{\alpha}}_\gamma f^{\tilde{\alpha}}_{\sigma}\eta_\sigma)^2] \\ 
	=& -K_2(tr\mathbf{A} \cdot \mathbf{A}_{\sigma\alpha}\eta_\sigma\eta_\alpha -(\mathbf{A}_{\gamma\sigma}\eta_\sigma)^2) \\
	=& -K_2(tr\mathbf{A} \cdot \lambda -\lambda^2).
\end{aligned}
\end{equation}
Then by \eqref{cla}, \eqref{ric1} and \eqref{sec}, we get 
\begin{equation} \label{11}
\frac{1}{2}\Delta_{\mathcal{H}}\tilde{\lambda} \ge -{K_1}\lambda + {K_2}\lambda(tr\mathbf{A}-\lambda). 
\end{equation}
From (\ref{qd}), (\ref{14}), (\ref{12}) and (\ref{11}), it follows that
\begin{equation} 
\begin{aligned}
0 &\ge \frac{\Delta_{\mathcal{H}}{\tilde{\lambda}}}{\lambda}-2\frac{\Delta_{\mathcal{H}}(r^2)}{a^2-r^2(x_0)}-\frac{6|{\nabla}^\mathcal{H}(a^2-r^2(x))|^2}{(a^2-r^2(x_0))^2} \\
&\ge 2(-K_1+K_2 (trA-\lambda))-\frac{24a^2}{{(a^2-r^2(x_0))^2}}-2\frac{2q(1+a\cdot C)+2}{a^2-r^2(x_0)}.
\end{aligned}
\end{equation}
As $f$ has bounded generalized transversal dilatation of order $\beta$, i.e., 
$$tr\mathbf{A}-\lambda \ge \frac{\lambda}{\beta^2}, $$ 
we get
\begin{equation} 
\lambda(x_0) \le \frac{K_1}{K_2}\beta^2+\frac{12a^2 \beta^2}{K_2{(a^2-r^2(x_0))^2}}+\frac{2q(1+a\cdot C)+2}{K_2(a^2-r^2(x_0))}\beta^2,
\end{equation} 
thus, for any  $x \in B_a (z)$ 
\begin{equation}
\begin{aligned}
(a^2-r^2(x))^2\lambda(x)&\le (a^2-r^2(x_0))^2\lambda(x_0) \\
&\le \frac{K_1\beta^2}{K_2}a^4+\frac{12a^2\beta^2}{K_2}+a^2\beta^2\frac{2q(1+a\cdot C)+2}{K_2}.
\end{aligned}
\end{equation}
Therefore,
\begin{equation}
\begin{aligned}
\lambda(x)&\le \frac{\beta^2}{K_2{(a^2-r^2(x))^2}}({K_1}a^4+{12a^2}+a^2(2q(1+a\cdot C)+2)).
\end{aligned}
\end{equation}
Let $a \to \infty$, we obtain
\begin{equation}
{\sup}_{M}\lambda \le \beta^2\frac{K_1}{K_2}.
\end{equation} 
\end{proof}
	
Next, we will establish a Schwarz type lemma for transversally holomorphic maps between K{\"a}hler foliations, which generalizes Yau's Schwarz lemma.

\begin{proof}[\textbf{Proof of Theorem \ref{thm1.2}}]
Let $u=e_{{\mathcal{H}},\tilde{\mathcal{H}}}(f)$, then it is clear that 
\begin{equation} \label{213}
f^*\tilde{g}_{\mathcal{\tilde{H}}}\le u g_{\mathcal{H}}.
\end{equation}
So it is sufficient to estimate the upper bound of $u$. For arbitrary point $z$ of M, we deﬁne the function 
\begin{equation}
\Phi(x)=(a^2-r^2(x))^2u(x).
\end{equation} 
Suppose $x_0$ is the maximum point of $\Phi$ in $B_a(z)$. The funcation $r(x)$ can be considered to be twice differentiable near $x_0$ (cf. \cite{ref8}). Applying the maximum principle to $\Phi$ at $x_0$, we have 
\begin{equation}
0=\frac{{\nabla}^\mathcal{H}u}{u} +2\frac{{\nabla}^\mathcal{H}(a^2-r^2)}{(a^2-r^2)},
\end{equation}
and 
\begin{equation} 
\begin{aligned} \label{qdd}
0\ge\frac{\Delta_{\mathcal{H}}u}{u} -2\frac{\Delta_{\mathcal{H}}r^2}{a^2-r^2}-6\frac{|{\nabla}^\mathcal{H}(a^2-r^2)|^2}{(a^2-r^2)^2}.
\end{aligned}	
\end{equation}
In terms of the curvature assumptions for both $M$ and $N$, the Bochner formula \eqref{bohoml} gives that 
\begin{equation} \label{x11}
\frac{1}{2}\Delta_{\mathcal{H}} u \ge -K_1u+K_2u^2.
\end{equation}
Similar to the argument in Theorem \ref{thm1.1}, by (\ref{14}), (\ref{12}), (\ref{qdd}) and (\ref{x11}), we deduce that
\begin{equation} 
u(x_0) \le \frac{K_1}{K_2}+\frac{12a^2}{K_2{(a^2-r^2(x_0))^2}}+\frac{2q(1+aC)+2}{K_2(a^2-r^2(x_0))}.
\end{equation}
Hence, 
\begin{equation}
\begin{aligned}
(a^2 -r^2(x))^2u(x) & \le(a^2-r^2(x_0))^2u(x_0)\\
		&\le \frac{K_1}{K_2}a^4+\frac{12a^2}{K_2}+a^2\frac{2q(1+aC)+2}{K_2},
\end{aligned}
\end{equation}
which implies 
\begin{equation}
\begin{aligned}
u(x)\le \frac{1}{K_2{(a^2-r^2(x))^2}}({K_1}a^4+{12a^2}+a^2(2q(1+aC)+2)),
\end{aligned}
\end{equation}
for any $x$ in $B_a(z)$. Let $a \to \infty$, we obtain that 
\begin{equation}
\sup_{x \in M}u(x) \le \frac{K_1}{K_2}.
\end{equation} 
\end{proof}

\section*{Acknowledgemnts}
The authors would like to thank Prof. Yuxin Dong and Prof. Gui Mu for their useful discussions and helpful comments.

\bigskip
Xin Huang

School of Mathematical Sciences

Fudan University

Shanghai 200433, P. R. China

17110180003@fudan.edu.cn

\bigskip

Weike Yu

School of mathematics and statistics

Nanjing University of Science and Technology

Nanjing, 210094, Jiangsu, P. R. China

wkyu2018@outlook.com

\bigskip

\end{document}